\documentclass[10pt]{amsart}

\usepackage[T1]{fontenc}
\usepackage{lmodern}
\usepackage{amsmath,amssymb,amsthm,stmaryrd}
\usepackage{mathtools}
\usepackage[margin=0.95in]{geometry}
\usepackage{xcolor}
\usepackage{microtype}
\usepackage{enumitem}

\newcommand{\Gal}{\operatorname{Gal}}
\newcommand{\gr}{\mathrm{gr}}
\newcommand{\Et}{\mathrm{\acute{E}t}}

\DeclareMathOperator{\Frac}{Frac}
\DeclareMathOperator{\Pic}{Pic}
\DeclareMathOperator{\Cl}{Cl}
\DeclareMathOperator{\Spec}{Spec}

\newcommand{\mot}{\operatorname{mot}}
\newcommand{\Fil}{\operatorname{Fil}}

\usepackage[colorlinks=true,linkcolor=blue!55!black,citecolor=blue!55!black,urlcolor=blue!55!black]{hyperref}
\usepackage{aliascnt}
\usepackage{tikz}
\usetikzlibrary{arrows.meta,positioning}

\usepackage[nameinlink,capitalize,noabbrev]{cleveref}
\newtheorem{theorem}{Theorem}[section]
\newaliascnt{proposition}{theorem}
\newtheorem{proposition}[proposition]{Proposition}
\aliascntresetthe{proposition}
\newaliascnt{lemma}{theorem}
\newtheorem{lemma}[lemma]{Lemma}
\aliascntresetthe{lemma}
\newaliascnt{corollary}{theorem}
\newtheorem{corollary}[corollary]{Corollary}
\aliascntresetthe{corollary}
\theoremstyle{remark}
\newaliascnt{remark}{theorem}
\newtheorem{remark}[remark]{Remark}
\aliascntresetthe{remark}
\newaliascnt{exercise}{theorem}

\aliascntresetthe{exercise}
\crefname{theorem}{Theorem}{Theorems}
\crefname{proposition}{Proposition}{Propositions}
\crefname{lemma}{Lemma}{Lemmas}
\crefname{corollary}{Corollary}{Corollaries}
\crefname{remark}{Remark}{Remarks}
\crefname{exercise}{Exercise}{Exercises}

\title{Finite-coefficient Gersten injectivity fails in ramified mixed characteristic}

\hypersetup{pdftitle={Finite-coefficient Gersten injectivity fails in ramified mixed characteristic}}


\author{Niels Feld}
\address{IRMAR
	\\
	University of Rennes
	\\
	263 Av. Général Leclerc
	\\
	35000 Rennes, France
}
\email{\href{mailto:niels.feld@univ-rennes.fr}{niels.feld@univ-rennes.fr}}
\urladdr{\url{https://nielsfeld.wixsite.com/website}}

\date{5 August 2026}

\begin{document}

\begin{abstract}
Let $V$ be a complete discrete valuation ring of mixed characteristic
$(0,3)$ in which $3$ is a uniformizer, and put
$A=V\llbracket x,y\rrbracket/(3+x^2-y^3)$.  We construct a nonzero class
$a\in K_2(A;\mathbf Z/3)$ whose restriction to the fraction field of $A$
is zero.  Thus Gersten injectivity for algebraic $K$-theory with
$\mathbf Z/3$-coefficients fails for a two-dimensional ramified regular local ring. 
We also indicate, without proof, the expected analogous statement in degree $p-1$ for every odd prime $p$.
This counterexample does not contradict the integral Gersten conjecture but it rules out a naive reduction to finite coefficients.
\end{abstract}
\maketitle

\tableofcontents

\section{Introduction}

For a ring $R$, write $K(R)$ for nonconnective algebraic $K$-theory and
\(K(R;\mathbf Z/3):=K(R)\wedge\mathbf S/3\), where \(\mathbf S/3\) is the
Moore spectrum. 
For every odd prime \(p\), we have \(p\cdot\operatorname{id}_{\mathbf S/p}=0\) in the
stable homotopy category. 
Hence the groups \(K_n(R;\mathbf Z/p)\) are
naturally \(\mathbf F_p\)-modules; in particular, multiplication by
\(1/2\) on a \(\mathbf Z/3\)-coefficient group means multiplication by
\(2\). We have \(K_n(R;\mathbf Z/3)=\pi_nK(R;\mathbf Z/3)\), and the
coefficient sequence contains
\begin{equation}\label{eq:bockstein-sequence}
	K_2(R)\xrightarrow{3}K_2(R)\longrightarrow K_2(R;\mathbf Z/3)
	\xrightarrow{\partial_R}K_1(R)\xrightarrow{3}K_1(R).
\end{equation}

Here finite coefficients mean the spectrum-level Moore construction
\(K(R)\wedge\mathbf S/3\), applied term by term to the Gersten row.
This should not be confused with derived tensoring a fixed integral
Gersten complex with \(\mathbf Z/3\).

Our main result is the following explicit counterexample.

\begin{theorem}\label{thm:main}
	Let $V$ be a complete discrete valuation ring of mixed characteristic
	$(0,3)$ in which $3$ is a uniformizer, and let $k$ be its residue field\footnote{In other words, $V$ is a Cohen ring \cite[Tag~0327]{StacksProject}; for instance $V=\mathbf Z_3$, or $V=W(k)$ for a perfect field $k$ of
		characteristic 3.}.
	Set
	\begin{equation}\label{eq:A-definition}
		A:=V\llbracket x,y\rrbracket/(3+x^2-y^3),\qquad F:=\Frac(A).
	\end{equation}
	Then $A$ is a complete two-dimensional ramified regular local domain and
	there is a class
	\[
	0\ne a\in\ker\bigl(K_2(A;\mathbf Z/3)\longrightarrow
	K_2(F;\mathbf Z/3)\bigr).
	\]
	Moreover \(K_1(A)[3]=K_1(F)[3]=0\), so this is also a failure
	of injectivity of \(K_2(A)/3\to K_2(F)/3\).
\end{theorem}

Consequently the augmented finite-coefficient Gersten row is not
exact at its first term.
By contrast, Gersten exactness with arbitrary
finite coefficients holds for unramified regular local rings
\cite[Cor.~5.7]{Skalit17}.  
This does not contradict the equicharacteristic integral results of
Quillen and Panin, Gillet's finite-coefficient theorem for discrete
valuation rings, or the smooth mixed-characteristic theorem of
Gillet-Levine
\cite{Quillen73,Panin03,Gillet86,GilletLevine87}.

 The ring in \eqref{eq:A-definition} is ramified: its maximal ideal is
 \((x,y)\) and
 \(
 3=y^3-x^2\in(x,y)^2.
 \)
 Its special fibre is the cusp
 \(
 A/3A\simeq k\llbracket x,y\rrbracket/(x^2-y^3),
 \)
 so \(A\) is regular but \(V\to A\) is not ind-smooth.
 Recent positive results
 for smooth or ind-smooth
 algebras over valuation rings likewise do not contain this example
\cite{BouisKunduBL,Kundu24SmoothValuation}.

There is one literature consequence worth making precise.  After
introducing an arbitrary cohomology theory with supports,
\cite[Conj.~1.1]{LuedersPAdic24} asserts that its Cousin complex resolves
the associated Zariski sheaf on every regular scheme.  In that literal
unrestricted form the conjecture is false: take the theory with supports
$H_Z^q(X):=K_{-q}^Z(X;\mathbf Z/3)$ and apply \cref{thm:main} to the local
scheme $\Spec(A)$.  The actual theorem of that paper concerns $p$-adic
\'etale Tate twists on schemes smooth over a henselian discrete valuation
ring and is unaffected \cite[Thm.~1.2]{LuedersPAdic24}.  Similarly, the
explicitly ``imprecise'' broad formulation in
\cite[Conj.~1.1]{Kundu24SmoothValuation} cannot be read as including mod-$p$ algebraic $K$-theory in residue characteristic $p$ on all
regular rings; the proved integral statements there remain unchanged.

The proof is short once two inputs are isolated.  The class itself is
elementary: the prime divisor $y=0$ of $\Spec(A)$ has coordinate ring
$A/yA\simeq O:=V[\zeta]$, and $a$ is the pushforward along
$\Spec(O)\hookrightarrow\Spec(A)$ of a mod-$3$ class of $O$.  
Its generic restriction vanishes because it is supported on a divisor.
Nonvanishing is proved after the flat base change
$A\to B[1/3]$, where $B$ is the normal surface $uv=y^3$ obtained from $A$ by
adjoining $\zeta$. After this base change, the divisor splits into two conjugate points and the class
becomes the product of a mod-$3$ class with a line bundle of order three in
$\Pic(B[1/3])$. This product is detected in weight two of the motivic filtration.
The motivic argument identifies the obstruction and controls its passage
from the associated graded to \(K\)-theory. It also isolates the
higher-degree pattern behind the proposed extension in
\cref{rem:odd-prime}.

\section{Motivic detection of Bott products with line-bundle classes}
\label{sec:motivic-detection}

We use the following consequence of the multiplicative motivic filtration
on algebraic \(K\)-theory.  In the equicharacteristic case relevant below,
the filtration and its motivic graded pieces were constructed by
Elmanto-Morrow
\cite[Thm.~1.1]{ElmantoMorrowEquicharMotivic};
Bouis extended this construction beyond equicharacteristic
\cite[Thm.~C]{BouisMixedCharacteristicMotivicCohomology}.
The filtered first Chern class and the ind-smooth comparison used below
are supplied by
\cite{BouisWeibelVanishingPBF,BouisKunduBL}.

\begin{lemma}
	\label{lem:filtered-bott-picard-compatibility}
	Let \(T\) be a connected regular affine noetherian scheme of dimension
	one, ind-smooth over a field of characteristic zero. Set
	\(F^q:=\Fil^q_{\mot}K(T)\) and
	\(F^q_3:= \Fil^q_{\mot}K(T;\mathbf Z/3) =\operatorname{cofib}(3\colon F^q\to F^q)\). Denote by
	\(i_q\colon F^q\to F^0\simeq K(T)\) and
	\(i_{q,3}\colon F^q_3\to F^0_3\simeq K(T;\mathbf Z/3)\)
	the structure maps of the two decreasing filtrations, and by
	\(e_q\colon F^q\to\gr^q_{\mot}K(T)\) and
	\(e_{q,3}\colon F^q_3\to\gr^q_{\mot}K(T;\mathbf Z/3)\)
	the quotient maps. Then the following assertions hold.
	\begin{enumerate}[label=\textup{(\roman*)}]
		\item The cofiber sequences defining the spectra \(F^q_3\) assemble
		into a cofiber sequence of filtered spectra
		\begin{equation}\label{eq:filtered-coefficient-triangle}
			\Fil^\bullet_{\mot}K(T)\xrightarrow{3}
			\Fil^\bullet_{\mot}K(T)\longrightarrow
			\Fil^\bullet_{\mot}K(T;\mathbf Z/3)
			\xrightarrow{\widetilde\partial}
			\Sigma\Fil^\bullet_{\mot}K(T).
		\end{equation}
		For every \(q\geq0\), applying \(\gr^q\) to
		\eqref{eq:filtered-coefficient-triangle} gives the motivic
		coefficient triangle
		\[
		\mathbf Z(q)_{\mot}(T)[2q]\xrightarrow{3}
		\mathbf Z(q)_{\mot}(T)[2q]\to
		\mathbf Z/3(q)_{\mot}(T)[2q]
		\xrightarrow{\delta_q[2q]}\mathbf Z(q)_{\mot}(T)[2q+1].
		\]
		Thus, on homotopy groups,
		\((e_q)_*\widetilde\partial_*
		=(\delta_q[2q])_*(e_{q,3})_*\).
		
		\item The map
		\((i_{1,3})_*\colon\pi_2F^1_3\to K_2(T;\mathbf Z/3)\)
		is an isomorphism. If \(\zeta\in\mu_3(T)\) and
		\(\beta\in K_2(T;\mathbf Z/3)\) satisfy
		\(\partial_T(\beta)=[\zeta]\), define
		\(\widetilde\beta:=(i_{1,3})_*^{-1}(\beta)\). Under the canonical
		identification
		\(
		\pi_2\gr^1_{\mot}K(T;\mathbf Z/3)
		=H^0_{\mot}(T,\mathbf Z/3(1))=\mu_3(T)
		\), one has \((e_{1,3})_*(\widetilde\beta)=\zeta\).
		
		\item The homomorphism \((i_1)_*\colon\pi_0F^1\to K_0(T)\)
		is injective with image \(\ker(\operatorname{rk}\colon K_0(T)\to
		\mathbf Z)\). For a line bundle \(\mathcal L\) on \(T\), let
		\(\widetilde\lambda\in\pi_0F^1\) be the unique class satisfying
		\((i_1)_*(\widetilde\lambda)=[\mathcal L]-1\). Then
		\((e_1)_*(\widetilde\lambda)=c_1^{\mot}(\mathcal L)\) in
		\(
		\pi_0\gr^1_{\mot}K(T)
		=H^2_{\mot}(T,\mathbf Z(1))
		\).
		
		\item The motivic-to-\'etale comparison maps
		\(
		\operatorname{comp}_q\colon
		\mathbf Z/3(q)_{\mot}(T)\to
		R\Gamma_{\Et}(T,\mu_3^{\otimes q})
		\)
		are multiplicative. Moreover,
		\(H^0(\operatorname{comp}_1)(\zeta)=\zeta\) and
		\(H^2(\operatorname{comp}_1)
		(c_1^{\mot}(\mathcal L)\bmod3)=\kappa(\mathcal L)\), where
		\(\kappa(\mathcal L)\in H^2_{\Et}(T,\mu_3)\) is the image of
		\([\mathcal L]\) under the boundary map of the Kummer sequence.
	\end{enumerate}
\end{lemma}

\begin{proof}
	By \cite[Defs.~3.18 and~3.20, Prop.~4.46]
	{BouisMixedCharacteristicMotivicCohomology}, the spectra \(F^q\)
	form a functorial multiplicative decreasing \(\mathbf N\)-indexed
	filtration of \(K(T)\), with \(F^0\simeq K(T)\) and
	\(
	\gr^q_{\mot}K(T)\simeq\mathbf Z(q)_{\mot}(T)[2q]
	\). Evaluation at a fixed filtration degree and the functor \(\gr^q\)
	are exact functors on filtered spectra. Hence the levelwise cofibers
	\(F^q_3\) form the cofiber sequence
	\eqref{eq:filtered-coefficient-triangle}, and
	\begin{equation}\label{eq:bott-picard-coefficient-graded}
		\gr^q_{\mot}K(T;\mathbf Z/3)
		\simeq\operatorname{cofib}
		\bigl(3\colon\gr^q_{\mot}K(T)\to
		\gr^q_{\mot}K(T)\bigr)
		\simeq\mathbf Z/3(q)_{\mot}(T)[2q].
	\end{equation}
	Exactness also identifies \(\gr^q(\widetilde\partial)\) with
	\(\delta_q[2q]\), including its commutative square with the quotient
	maps. The multiplication gives maps
	\(F^r_3\wedge F^s\to F^{r+s}_3\). Under
	\eqref{eq:bott-picard-coefficient-graded}, their associated-graded maps
	are
	\[
	\mathbf Z/3(r)_{\mot}(T)[2r]\otimes^{\mathbf L}
	\mathbf Z(s)_{\mot}(T)[2s]\longrightarrow
	\mathbf Z/3(r+s)_{\mot}(T)[2r+2s],
	\]
	induced by the multiplication of the motivic complexes. This proves
	\textup{(i)}.
	
	Choose a field \(k_0\) of characteristic zero over which \(T\) is
	ind-smooth. The classical-motivic comparison is an equivalence on \(T\)
	by \cite[Thm.~6.1]{BouisKunduBL}. Thus
	\(
	\mathbf Z(0)_{\mot}(T)\simeq
	R\Gamma_{\mathrm{Zar}}(T,\mathbf Z)
	\) and
	\(
	\mathbf Z(1)_{\mot}(T)\simeq
	R\Gamma_{\mathrm{Zar}}(T,\mathbf G_m)[-1]
	\); see also \cite[Examples~3.4 and~3.5]
	{BouisMixedCharacteristicMotivicCohomology}.
	
	We record the weight-zero consequence in the degrees that will be used.
	Since \(T\) is connected and regular, it is integral. The
	constant Zariski sheaves \(\mathbf Z\) and \(\mathbf Z/3\) on \(T\)
	are consequently flasque, and their cohomology is concentrated in degree
	zero. 
	Consequently
	\(
	\mathbf Z(0)_{\mot}(T)\simeq\mathbf Z,
	\) and \(
	\mathbf Z/3(0)_{\mot}(T)\simeq\mathbf Z/3
	\)
	concentrated in degree zero. It follows from
	\eqref{eq:bott-picard-coefficient-graded} that, for \(j>0\), \(
	\pi_j\gr^0_{\mot}K(T;\mathbf Z/3)=0
	.
	\)
	Under the classical-motivic comparison, the map
	\(
	\pi_0F^0=K_0(T)\to\pi_0\gr^0_{\mot}K(T)=\mathbf Z
	\)
	is the usual
	rank homomorphism. The fibre sequence \(F^1_3\to F^0_3\to
	\gr^0_{\mot}K(T;\mathbf Z/3)\) gives the exact segment
	\[
	0=\pi_3\gr^0_{\mot}K(T;\mathbf Z/3)\to
	\pi_2F^1_3\xrightarrow{(i_{1,3})_*}K_2(T;\mathbf Z/3)\to
	\pi_2\gr^0_{\mot}K(T;\mathbf Z/3)=0.
	\] 
	Thus \((i_{1,3})_*\) is an isomorphism.
	
	The integral weight-zero fibre sequence similarly gives
	\((i_1)_*\colon\pi_1F^1\xrightarrow{\sim}K_1(T)\), because
	\(\pi_2\gr^0_{\mot}K(T)=\pi_1\gr^0_{\mot}K(T)=0\). Define the
	weight-one edge homomorphism
	\(\epsilon_1\colon K_1(T)\to H^1_{\mot}(T,\mathbf Z(1))\) as
	\((e_1)_*\circ(i_1)_*^{-1}\). 
	The compatibility between the filtered
	orientation and the motivic first Chern class, used in
	\cite[Rem.~4.4]{BouisWeibelVanishingPBF} and made explicit in
	\cite[Lemma 5.16]{ElmantoMorrowEquicharMotivic}, identifies the restriction of
	\(\epsilon_1\) along
	\(\Gamma(T,\mathcal O_T)^\times\to K_1(T)\) with the map on \(H^1\)
	induced by
	\[
	c_1^{\mot}\colon
	R\Gamma_{\mathrm{Nis}}(T,\mathbf G_m)[-1]\longrightarrow
	\mathbf Z(1)_{\mot}(T)
	\]
	of \cite[Def.~4.1 and Rem.~4.2]{BouisWeibelVanishingPBF} (see also \cite[Lemma 5.1 and Definition 5.2]{ElmantoMorrowEquicharMotivic}). Under the
	classical-motivic comparison on \(T\), this map on \(H^1\) is the
	identity of \(\Gamma(T,\mathcal O_T)^\times\). In particular,
	\(\epsilon_1([\zeta])=\zeta\).
	
	Naturality of the boundary in
	\eqref{eq:filtered-coefficient-triangle} gives
	\((i_1)_*\widetilde\partial_*(\widetilde\beta)
	=\partial_T(\beta)=[\zeta]\). Therefore
	\(\widetilde\partial_*(\widetilde\beta)\) is the unique element of
	\(\pi_1F^1\) mapped to \([\zeta]\) by \((i_1)_*\).
	
	Put \(\bar\beta:=(e_{1,3})_*(\widetilde\beta)\). Part \textup{(i)}
	and the preceding equality give
	\(\delta_1(\bar\beta)=\zeta\) in
	\(H^1_{\mot}(T,\mathbf Z(1))=\Gamma(T,\mathcal O_T)^\times\).
	Under \(\mathbf Z(1)_{\mot}(T)\simeq
	R\Gamma_{\mathrm{Zar}}(T,\mathbf G_m)[-1]\), the long exact sequence
	of the weight-one coefficient triangle identifies
	\(H^0_{\mot}(T,\mathbf Z/3(1))\) with \(\mu_3(T)\), and identifies
	\(\delta_1\) with the inclusion
	\(\mu_3(T)\hookrightarrow\Gamma(T,\mathcal O_T)^\times\).
	Consequently \(\bar\beta=\zeta\), proving \textup{(ii)}.
	
	The integral weight-zero fibre sequence also gives the exact sequence
	\(
	0=\pi_1\gr^0_{\mot}K(T)\to\pi_0F^1
	\xrightarrow{(i_1)_*}K_0(T)\xrightarrow{\operatorname{rk}}\mathbf Z
	\). This proves the first assertion of \textup{(iii)} and defines
	\(\widetilde\lambda\) without ambiguity.
	
	It remains to calculate \((e_1)_*(\widetilde\lambda)\). Since \(T\) is
	affine, integral, noetherian, regular, and one-dimensional, its
	coordinate ring is a Dedekind domain. Every invertible module over a
	Dedekind domain is generated by two elements. Choosing two generators of
	\(\mathcal L\) gives a quotient \(\mathcal O_T^{\oplus2}\twoheadrightarrow
	\mathcal L\), hence a morphism \(f\colon T\to\mathbf P^1_{k_0}\) such that
	\(f^*\mathcal O_{\mathbf P^1_{k_0}}(1)\simeq\mathcal L\).
	
	We now apply \cite[Constr.~4.3 and Rem.~4.4]
	{BouisWeibelVanishingPBF} with \(X=\Spec({k_0})\) and \(E=K\). If
	\(\pi\colon\mathbf P^1_{k_0}\to\Spec({k_0})\) is the projection, the second
	summand of the filtered projective-bundle map is a morphism
	\(
	u\colon\Fil^{\bullet-1}_{\mot}K({k_0})\to
	\Fil^\bullet_{\mot}K(\mathbf P^1_{k_0})
	\). At filtration degree one it gives
	\(u_1\colon\Fil^0_{\mot}K({k_0})\to
	\Fil^1_{\mot}K(\mathbf P^1_{k_0})\). For \(K\)-theory, the line-bundle
	operator in Construction~4.3 is induced by the exact autoequivalence
	\(-\otimes\mathcal O(-1)\). Hence the composite
	\(
	\pi_0\Fil^0_{\mot}K({k_0})\xrightarrow{(u_1)_*}
	\pi_0\Fil^1_{\mot}K(\mathbf P^1_{k_0})\to K_0(\mathbf P^1_{k_0})
	\)
	sends \(1\) to \(1-[\mathcal O(-1)]\).
	
	On associated graded pieces, \cite[Rem.~4.4]
	{BouisWeibelVanishingPBF} identifies \(\gr^1(u)\), after the shifts
	\(\gr^j_{\mot}K\simeq\mathbf Z(j)_{\mot}[2j]\), with the second map in
	\cite[(4.2.1)]{BouisWeibelVanishingPBF}. By
	\cite[Def.~4.1 and Rem.~4.2]{BouisWeibelVanishingPBF}, this map is
	\[
	c_1^{\mot}(\mathcal O(1))\pi^*\colon
	\mathbf Z(0)_{\mot}({k_0})\to
	\mathbf Z(1)_{\mot}(\mathbf P^1_{k_0})[2].
	\]
	Therefore, if
	\(\eta:=(u_1)_*(1)\in
	\pi_0\Fil^1_{\mot}K(\mathbf P^1_{k_0})\), then its image under
	\(\pi_0\Fil^1_{\mot}K(\mathbf P^1_{k_0})\to
	\pi_0\gr^1_{\mot}K(\mathbf P^1_{k_0})\) is
	\(c_1^{\mot}(\mathcal O(1))\).	
	Functoriality of the motivic filtration gives
	\(\eta_{\mathcal L}:=f^*(\eta)\in\pi_0F^1\). The two calculations above and
	naturality of \(c_1^{\mot}\) give
	\((i_1)_*(\eta_{\mathcal L})=1-[\mathcal L^\vee]\) and
	\((e_1)_*(\eta_{\mathcal L})=c_1^{\mot}(\mathcal L)\). Let
	\(b_{\mathcal L}:=[\mathcal L]\cdot \eta_{\mathcal L}\in\pi_0F^1\),
	where the product is the map
	\(\pi_0F^0\otimes\pi_0F^1\to\pi_0F^1\) induced by the multiplicative
	motivic filtration. Then
	\((i_1)_*(b_{\mathcal L})=[\mathcal L](1-[\mathcal L^\vee])
	=[\mathcal L]-1\), so injectivity of \((i_1)_*\) gives
	\(b_{\mathcal L}=\widetilde\lambda\). The image of \([\mathcal L]\)
	in \(\pi_0\gr^0_{\mot}K(T)=H^0_{\mot}(T,\mathbf Z(0))=\mathbf Z\)
	is its rank, namely \(1\). Multiplicativity on associated graded pieces
	therefore gives \((e_1)_*(\widetilde\lambda)
	=c_1^{\mot}(\mathcal L)\). This proves \textup{(iii)}.
	
	Finally, in the equicharacteristic situation at hand the
	Beilinson-Lichtenbaum map is a map of graded
	\(\mathbf E_\infty\)-algebras and is compatible with the first Chern
	class; see
	\cite[Thm.~1.1 and Remark 4.42]{ElmantoMorrowEquicharMotivic} and
	\cite[Thm.~7.16]{BachmannElmantoMorrowCdhMotivic}. In weight one,
	\(H^0(\operatorname{comp}_1)\colon\mu_3(T)\to
	H^0_{\Et}(T,\mu_3)=\mu_3(T)\) is the identity, while
	\(H^2(\operatorname{comp}_1)\) sends the motivic first Chern class
	modulo \(3\) to the boundary class of \(\mathcal L\) for
	\(1\to\mu_3\to\mathbf G_m\xrightarrow{3}\mathbf G_m\to1\).
	These are precisely the two identities in \textup{(iv)}, and
	multiplicativity gives
	\(
	\operatorname{comp}_2
	(\zeta\smile(c_1^{\mot}(\mathcal L)\bmod3))
	=\zeta\smile\kappa(\mathcal L)
	\). This proves \textup{(iv)}.
\end{proof}

\begin{lemma}
	\label{lem:bott-picard}
	Let \(T\) be a connected regular affine noetherian scheme of dimension
	one, ind-smooth over a field of characteristic zero. Suppose that
	\(\Gamma(T,\mathcal O_T)\) contains a primitive cube root of unity
	\(\zeta\). Let \(\mathcal L\) be a line bundle whose class is nonzero in
	\(\Pic(T)/3\), and let \(\beta\in K_2(T;\mathbf Z/3)\) satisfy
	\(\partial_T(\beta)=[\zeta]\in K_1(T)[3]\). Then
	\(\beta([\mathcal L]-1)\ne0\) in \(K_2(T;\mathbf Z/3)\).
\end{lemma}

\begin{proof}
	Use the notation \(F^q,F^q_3,i_q,i_{q,3},e_q,e_{q,3}\) of
	\cref{lem:filtered-bott-picard-compatibility}. Define
	\(\widetilde\beta:=(i_{1,3})_*^{-1}(\beta)\in\pi_2F^1_3\), and let
	\(\widetilde\lambda\in\pi_0F^1\) be the unique class such that
	\((i_1)_*(\widetilde\lambda)=[\mathcal L]-1\). By parts
	\textup{(ii)} and \textup{(iii)} of that lemma,
	\((e_{1,3})_*(\widetilde\beta)=\zeta\) and
	\((e_1)_*(\widetilde\lambda)=c_1^{\mot}(\mathcal L)\).
	
	The module multiplication of part \textup{(i)}, at bidegree
	\((1,1)\), is a map \(F^1_3\wedge F^1\to F^2_3\); it defines
	\(\widetilde\beta\,\widetilde\lambda\in\pi_2F^2_3\). Its image
	under \((i_{2,3})_*\) is
	\(\beta([\mathcal L]-1)\). The induced pairing on associated graded
	pieces is the motivic cup product, so
	\begin{equation}\label{eq:bott-picard-product-leading-term}
		(e_{2,3})_*(\widetilde\beta\,\widetilde\lambda)=
		\zeta\smile\bigl(c_1^{\mot}(\mathcal L)\bmod3\bigr)
		\in H^2_{\mot}(T,\mathbf Z/3(2)).
	\end{equation}
	
	The long exact sequence of the weight-one coefficient triangle contains
	\(
	\Pic(T)\xrightarrow{3}\Pic(T)\to
	H^2_{\mot}(T,\mathbf Z/3(1))
	\). Hence the second map induces an injection
	\(\Pic(T)/3\hookrightarrow H^2_{\mot}(T,\mathbf Z/3(1))\), and the
	hypothesis on \(\mathcal L\) gives
	\(c_1^{\mot}(\mathcal L)\bmod3\ne0\).
	
	Since \(3\) is invertible on \(T\), Beilinson-Lichtenbaum gives an
	isomorphism
	\(
	H^2_{\mot}(T,\mathbf Z/3(2))\xrightarrow{\sim}
	H^2_{\Et}(T,\mu_3^{\otimes2})
	\); see \cite[Cor.~5.6]{BouisMixedCharacteristicMotivicCohomology}.
	By part \textup{(iv)} of
	\cref{lem:filtered-bott-picard-compatibility}, the image of
	\eqref{eq:bott-picard-product-leading-term} under this isomorphism is
	\(\zeta\smile\kappa(\mathcal L)\).
	
	The Kummer sequence gives the exact segment
	\(
	\Pic(T)\xrightarrow{3}\Pic(T)\xrightarrow{\kappa}
	H^2_{\Et}(T,\mu_3)
	\), so \(\kappa\) induces an injection
	\(\Pic(T)/3\hookrightarrow H^2_{\Et}(T,\mu_3)\). In particular,
	\(\kappa(\mathcal L)\ne0\). Since \(\zeta\) is primitive, the map of
	\'etale sheaves
	\(
	\mu_3\to\mu_3^{\otimes2},\quad s\mapsto\zeta\otimes s
	\)
	is an isomorphism. The induced map on \(H^2_{\Et}(T,-)\) sends
	\(\kappa(\mathcal L)\) to \(\zeta\smile\kappa(\mathcal L)\).
	Consequently \(\zeta\smile\kappa(\mathcal L)\ne0\), and
	\eqref{eq:bott-picard-product-leading-term} implies
	\((e_{2,3})_*(\widetilde\beta\,\widetilde\lambda)\ne0\). Thus \(\widetilde\beta\,\widetilde\lambda\ne0\) in \(\pi_2F^2_3\).
	It remains to verify that the image of \(\widetilde\beta\,\widetilde\lambda\) in
	\(K_2(T;\mathbf Z/3)\) is nonzero. The weight-one identification
	\(\mathbf Z(1)_{\mot}(T)\simeq
	R\Gamma_{\mathrm{Zar}}(T,\mathbf G_m)[-1]\) and its coefficient
	triangle give
	\(H^{-1}_{\mot}(T,\mathbf Z/3(1))=0\). By
	\eqref{eq:bott-picard-coefficient-graded}, this is the equality
	\(\pi_3\gr^1_{\mot}K(T;\mathbf Z/3)=0\). The fibre sequence
	\(F^2_3\to F^1_3\to\gr^1_{\mot}K(T;\mathbf Z/3)\) therefore shows
	that \(\pi_2F^2_3\to\pi_2F^1_3\) is injective. Part \textup{(ii)} of
	\cref{lem:filtered-bott-picard-compatibility} identifies
	\(\pi_2F^1_3\) with \(K_2(T;\mathbf Z/3)\), and the image of \(\widetilde\beta\,\widetilde\lambda\)
	under this composite injection is \(\beta([\mathcal L]-1)\). Hence
	\(\beta([\mathcal L]-1)\ne0\), as claimed. Notice that the argument
	uses neither completeness nor separatedness of the motivic filtration.
\end{proof}

\section{The cyclotomic \texorpdfstring{$A_2$}{A2} surface}

Choose a primitive cube root of unity \(\zeta\), put
\(\pi:=\zeta-1\), and set
\(
O:=V[\zeta]=V[\pi].
\)
The element \(\pi\) satisfies
\(
\pi^2+3\pi+3=0.
\)
Since \(3\) is a uniformizer of \(V\), this equation is Eisenstein.
Hence \(O\) is a discrete valuation ring, finite free and totally ramified
of degree two over \(V\), with maximal ideal \((\pi)\).  Its nontrivial
automorphism is denoted by \(\sigma\).  Set
\[
 t:=\zeta^2-\zeta,\qquad u:=x-t,\qquad v:=x+t.
\]
Since $t^2=-3$, base change of \eqref{eq:A-definition} gives
\begin{equation}\label{eq:B-definition}
 B:=A\otimes_VO\simeq O\llbracket x,y\rrbracket/(uv-y^3).
\end{equation}
Let $P:=(u,y)$ and $Q:=(v,y)$ be the height-one prime ideals of $B$.
Put \(R:=B[1/3]\) and \(T:=\Spec(R)\).

Our divisor-class convention is that \([P]\in\Cl(B)\) denotes the class of
the prime Weil divisor \(V(P)\). By contrast, \(P|_T\) denotes the
restriction to \(T\) of the rank-one ideal sheaf \(P\), which corresponds
to \(\mathcal O_T(-V(P)|_T)\). Thus the two conventions differ by the
usual sign when divisor classes are identified with rank-one reflexive
sheaves.

\begin{proposition}\label{prop:geometry}
The ring $B$ is a complete normal local domain, finite free of rank two
over $A$.  
Its divisor class group is cyclic of order three, generated by
$[P]$, \( \operatorname{div}(u)=3P\), and \( \operatorname{div}(y)=P+Q\).
If \(\mathcal L:=P|_T\), then \(T\) is integral, regular
and one-dimensional, $\Pic(T)\simeq\mathbf Z/3$, the class of $\mathcal L$
is a generator, and $Q|_T\simeq\mathcal L^{-1}$.

\end{proposition}
\begin{proof}
	Let	$\mathfrak n:=(3,x,y)$ in $V\llbracket x,y\rrbracket$.  The ring $V\llbracket x,y\rrbracket$ is a three-dimensional complete
	regular local ring, and
	since \(3\) is a uniformizer of \(V\), one has
	\(3\notin\mathfrak n^2\), and
	$3+x^2-y^3\equiv3\pmod{\mathfrak n^2}$.  Thus the defining equation is
	part of a minimal system of generators of $\mathfrak n$, so
	\(A=V\llbracket x,y\rrbracket/(3+x^2-y^3)\) is a two-dimensional regular local domain
	\cite[Tags~032C, 00NQ and~00NP]{StacksProject}.  As a quotient of the complete
	noetherian local ring \(V\llbracket x,y\rrbracket\), the ring \(A\) is complete.  Its maximal
	ideal is \((x,y)\),
	and $3=y^3-x^2\in(x,y)^2$, which is precisely the asserted ramification.
	Since $O$ is free of rank two over $V$, the identity
	$B=A\otimes_VO$ makes $B$ finite free of rank two over $A$.  The
	presentation \eqref{eq:B-definition} makes it a complete local ring
	\cite[Tag~0325]{StacksProject}.
	For this uniformizer one has \(t=\zeta\pi\) and
	\(\pi^2=-3\zeta\). Since $A$ is a domain of characteristic zero, it is
	$V$-torsion-free and hence flat over the discrete valuation ring $V$
	\cite[Tag~0539]{StacksProject}.  Consequently $B=A\otimes_VO$ is
	$O$-flat, so $\pi$ is a nonzerodivisor.  Reduction modulo $\pi$ gives
	\[
	B/\pi B\simeq k\llbracket x,y\rrbracket/(x^2-y^3)
	\simeq k\llbracket s^2,s^3\rrbracket,
	\qquad x\longmapsto s^3,\quad y\longmapsto s^2,
	\]
	which is a domain.  
	
	Krull's intersection theorem gives \(\bigcap_n\mathfrak m_B^n=0\), and
	\(\pi^nB\subseteq\mathfrak m_B^n\), so every nonzero \(b\in B\) lies in
	\(\pi^rB\) for a largest \(r\ge0\).  If \(b=\pi^rb_0\) and
	\(c=\pi^sc_0\) are nonzero with \(b_0,c_0\notin\pi B\) and \(bc=0\), then
	cancelling the nonzerodivisor \(\pi^{r+s}\) gives \(b_0c_0=0\), whose
	reduction modulo \(\pi\) contradicts the display above.  Hence \(B\) is a
	domain.

	Write
	\(
	B\simeq O\llbracket u,y\rrbracket/(u^2+2tu-y^3).
	\)
	The ambient ring is regular local of dimension three and the displayed
	equation is a nonzerodivisor.  Hence $B$ is a two-dimensional
	Cohen-Macaulay hypersurface, and in particular satisfies $S_2$
	\cite[Tags~00NQ, 02JN and~0342]{StacksProject}.

	Moreover
	\(B[1/\pi]\simeq
	A[1/3]\otimes_{V[1/3]}O[1/3]\).
	The ring \(A[1/3]\) is regular, while
	\(O[1/3]/V[1/3]\) is finite \'etale; regularity is preserved by
	localization and \'etale base change
	\cite[Tags~0AFS and~0AH0]{StacksProject}.  Hence \(B[1/\pi]\) is
	regular.  It is one-dimensional: every prime not containing \(\pi\)
	has height at most one, since the unique height-two prime
	\(\mathfrak m_B\) contains \(\pi\), while \(P\) survives after inverting
	\(\pi\), because \(B/P\simeq O\) and \(\pi\ne0\) in \(O\).

	Since $B/\pi B$ is a
	domain, $(\pi)$ is prime; it has height one, and
	$B_{(\pi)}$ is one-dimensional with maximal ideal generated by $\pi$,
	so it is a discrete valuation ring.  Therefore $B$ satisfies $R_1$ and
	$S_2$, and Serre's criterion shows that it is normal
	\cite[Tag~031O]{StacksProject}.

		We next compute \(\Cl(B)\).  Since
	\(B/(u)\simeq O\llbracket y\rrbracket/(y^3)\), the only height-one
	prime containing \(u\) is \(P\).  In \(B_P\), the elements \(t\) and
	\(v=u+2t\) are units, \(PB_P=(y)\), and \(u=y^3/v\).  Thus
	\(\operatorname{div}(u)=3P\).  Similarly,
	\(B/(y)\simeq O\llbracket u\rrbracket/(u(u+2t))\), and localization
	at either of its two minimal primes shows that
	\(\operatorname{div}(y)=P+Q\).  Hence \(3[P]=0\) and
	\([P]+[Q]=0\).
	Moreover, \([P]\ne0\).  Indeed, if \(P\) were principal, its image
	modulo \(\pi\) would be principal.  Since \(u=x-t\) and \(t=\zeta\pi\),
	this image is the maximal ideal
	\((\bar x,\bar y)\) of
	\(k\llbracket x,y\rrbracket/(x^2-y^3)\), whereas
	\((\bar x,\bar y)/(\bar x,\bar y)^2\) has \(k\)-dimension two.
	Thus \([P]\) has exact order three.
	
	It remains to prove that \(P\) generates the whole class group.\footnote{This is the direct \(A_2\) instance of Lipman's
		discriminant-group description of class groups of rational surface
		singularities; compare
		\cite[Prop.~17.1]{LipmanRationalSingularities}.}
	Let
	\(\rho\colon\widetilde X\to\Spec(B)\) be the blow-up of
	\(\mathfrak m_B=(\pi,x,y)\); it is projective, hence proper \cite[Tags~02NS/01WC]{StacksProject}.  It is an
	isomorphism over \(\Spec(B)\setminus\{\mathfrak m_B\}\), which is
	regular because \(B\) is two-dimensional and satisfies \(R_1\); so only
	the points of the exceptional divisor have to be examined.  Since
	\(\Spec(B)\) is a hypersurface in \(\Spec(O\llbracket x,y\rrbracket)\)
	and the strict transform of a closed subscheme is its blow-up in the
	induced centre \cite[Tag~080E]{StacksProject}, \(\widetilde X\) is the
	strict transform of \(\Spec(B)\) in the blow-up of
	\(\Spec(O\llbracket x,y\rrbracket)\) at its maximal ideal, which is
	regular of dimension three; in each chart below \(\widetilde X\) is
	therefore cut out by one equation in a three-dimensional regular local
	ring.  Since
	\(\pi^2=-3\zeta\), the equation of \(B\) is
	\(x^2-\zeta^2\pi^2-y^3=0\), whose initial form is \(X^2-\Pi^2\), since \(\zeta\equiv1\bmod\pi\).  This
	form is a nonzerodivisor in \(k[X,\Pi,Y]\), so it generates the initial
	ideal and
	\(\operatorname{gr}_{\mathfrak m_B}(B)\simeq k[X,\Pi,Y]/(X^2-\Pi^2)\).
	The exceptional divisor is therefore the reduced
	union \(E=E_1\cup E_2\) of the lines \(X=\Pi\) and \(X=-\Pi\) in
	\(\mathbf P^2_k\), meeting transversally at
	\(e=[0:0:1]\).
	
	On the \(\Pi\)-chart the strict transform has equation
	\((x_1-\zeta)(x_1+\zeta)=\pi y_1^3\).  Along the exceptional divisor
	exactly one of the two factors lies in the maximal ideal, the other
	being a unit because \(2\zeta\) is a unit; hence the equation is part of a
	regular system of parameters and the local ring is regular.  This chart
	contains \(E\setminus\{e\}\).
	
	At \(e\) we pass to the \(Y\)-chart, where \(x=yx_3\) and
	\(\pi=y\pi_3\); eliminating \(y\) by the equation
	\(y=x_3^2-\zeta^2\pi_3^2\) of the strict transform, the completed local
	ring there is
	\(O\llbracket x_3,\pi_3\rrbracket/(\pi-(x_3^2-\zeta^2\pi_3^2)\pi_3)\),
	which is regular because its defining equation is congruent to \(\pi\)
	modulo the square of the maximal ideal.  Hence \(\widetilde X\) is
	regular \cite[Tag~07NY]{StacksProject}.
	
	In that chart, \(E\) is cut out by
	\(y=(x_3-\zeta\pi_3)(x_3+\zeta\pi_3)\), and
	\((x_3-\zeta\pi_3,x_3+\zeta\pi_3)=(x_3,\pi_3)\) is the maximal ideal;
	so \(E_1\) and \(E_2\) meet transversally at \(e\).  On the
	\(\Pi\)-chart they are disjoint, because \(2\zeta\) is a unit.  Hence the
	scheme-theoretic intersection \(E_1\cap E_2\) is the reduced point
	\(e\), whose residue field is \(k\).

	Each \(E_i\) is a line in \(\mathbf P^2_k\), hence a proper \(k\)-scheme
	of dimension one, so every invertible \(\mathcal O_{E_i}\)-module
	\(\mathcal N\) has a degree
	\(\deg_{E_i}(\mathcal N)=\chi(E_i,\mathcal N)-\chi(E_i,\mathcal O_{E_i})\)
	\cite[Tag~0AYR]{StacksProject}, and \(\deg_{E_i}\) is a homomorphism
	\(\Pic(E_i)\to\mathbf Z\) \cite[Tag~0AYX]{StacksProject}.  For a Cartier
	divisor \(D\) on \(\widetilde X\) put
	\(D\cdot E_i:=\deg_{E_i}\bigl(\mathcal O_{\widetilde X}(D)|_{E_i}\bigr)\),
	which is therefore additive in \(D\).
	
	By construction of the blow-up, \(\mathfrak m_B\mathcal O_{\widetilde X}\)
	is the ideal of the effective Cartier divisor \(E\), and
	\(\mathcal O_{\widetilde X}(-E)=\mathcal O_{\widetilde X}(1)\)
	\cite[Tag~02OS]{StacksProject}.  Under the closed immersion
	\(E=\operatorname{Proj}(\operatorname{gr}_{\mathfrak m_B}(B))
	\subset\mathbf P^2_k\) one has
	\(\mathcal O_{\widetilde X}(1)|_E=\mathcal O_{\mathbf P^2}(1)|_E\), whose
	restriction to the line \(E_i\simeq\mathbf P^1_k\) is
	\(\mathcal O_{\mathbf P^1}(1)\), of degree one.  Hence
	\(E\cdot E_i=-1\) for \(i=1,2\).
	
	Moreover \(E_2\not\subset E_1\), so the pullback of \(E_1\) along
	\(E_2\hookrightarrow\widetilde X\) is an effective Cartier divisor on
	\(E_2\) and
	\(\mathcal O_{\widetilde X}(E_1)|_{E_2}=\mathcal O_{E_2}(E_1\cap E_2)\).
	By the chart computations above, \(E_1\cap E_2\) is the reduced point
	\(e\) with residue field \(k\), so
	\(E_1\cdot E_2=\dim_k\Gamma(E_1\cap E_2,\mathcal O)=1\)
	\cite[Tag~0AYY]{StacksProject}; exchanging the roles of \(E_1\) and
	\(E_2\) gives \(E_2\cdot E_1=1\).  Finally \(E\) is reduced with
	irreducible components \(E_1\) and \(E_2\), and \(\widetilde X\) is
	regular, so \(E=E_1+E_2\) as Cartier divisors.  Additivity in the first
	variable therefore yields
	\(E_1^2=E\cdot E_1-E_2\cdot E_1=-2\) and
	\(E_2^2=E\cdot E_2-E_1\cdot E_2=-2\).

	For \(n\ge1\) let \(E^{(n)}\subset\widetilde X\) be the closed subscheme
	defined by \(\mathcal O_{\widetilde X}(-nE)\), so that \(E^{(1)}=E\) and
	\(E^{(n)}=\widetilde X\times_{\Spec(B)}\Spec(B/\mathfrak m_B^n)\).  Let
	\(\nu\colon E_1\sqcup E_2\to E\) be the normalization.  The sequence
	\[
	1\to\mathcal O_E^\times\to\nu_*\mathcal O^\times\to k^\times(e)\to1
	\]
	is exact and \(H^0(\nu_*\mathcal O^\times)=k^\times\times k^\times\to
	k^\times(e)\) is surjective, so multidegree identifies \(\Pic(E)\) with
	\(\mathbf Z^2\).  
	
	Tensoring the additive normalization sequence
	\(0\to\mathcal O_E\to\nu_*\mathcal O\to k(e)\to0\) with
	\(\mathcal O_E(-nE)\) gives a sequence whose restrictions to the two
	components are \(\mathcal O_{\mathbf P^1}(n)\).  Since evaluation at
	\(e\) is surjective on global sections and
	\(H^1(\mathbf P^1,\mathcal O(n))=0\) for \(n\ge0\), it follows that
	\(H^1(E,\mathcal O_E(-nE))=0\).  Also
	\(H^2(E,\mathcal O_E(-nE))=0\), since \(E\) is a curve. For \(n\ge1\), the morphism
	\(E^{(n)}\to E^{(n+1)}\) is a square-zero thickening with ideal
	\(\mathcal I_n:=
	\mathcal O_{\widetilde X}(-nE)/
	\mathcal O_{\widetilde X}(-(n+1)E)
	\simeq\mathcal O_E(-nE)\).
	The exact sequence of Zariski sheaves
	\[
	1\longrightarrow1+\mathcal I_n\longrightarrow
	\mathcal O_{E^{(n+1)}}^\times\longrightarrow
	\mathcal O_{E^{(n)}}^\times\longrightarrow1,
	\qquad 1+\mathcal I_n\simeq\mathcal I_n,
	\]
	gives an exact segment
	\(H^1(E,\mathcal I_n)\to\Pic(E^{(n+1)})\to
	\Pic(E^{(n)})\to H^2(E,\mathcal I_n)\)
	\cite[Tag~0C6R]{StacksProject}. The two outer groups vanish by the
	previous paragraph, so
	\(\Pic(E^{(n+1)})\simeq\Pic(E^{(n)})\).
	The same vanishing of \(H^1(E,\mathcal I_n)\) makes the restriction
	maps on global units surjective, so isomorphisms between compatible
	systems of line bundles can be chosen compatibly.

	Let
	\(
	(\mathcal L_n)_{n\geq1}
	\)
	be a compatible system with
	\(\mathcal L_n\in\Pic(E^{(n)})\). Since \(B\) is noetherian and
	\(\mathfrak m_B\)-adically complete and
	\(\rho\colon\widetilde X\to\Spec(B)\) is proper, Grothendieck
	existence identifies the completion functor
	\[
	\operatorname{Coh}(\widetilde X)
	\xrightarrow{\ \sim\ }
	\operatorname{Coh}
	\bigl(\widetilde X,
	\mathfrak m_B\mathcal O_{\widetilde X}\bigr),
	\qquad
	\mathcal F\longmapsto
	\bigl(\mathcal F/\mathfrak m_B^n\mathcal F\bigr)_{n\geq1};
	\]
	see \cite[Tag~088C]{StacksProject}. Here
	\(
	\mathfrak m_B^n\mathcal O_{\widetilde X}
	=\mathcal O_{\widetilde X}(-nE)
	\),
	so the \(n\)-th quotient is a coherent sheaf on \(E^{(n)}\).
	Consequently the system \((\mathcal L_n)_n\) is the completion of a
	coherent \(\mathcal O_{\widetilde X}\)-module \(\mathcal F\), unique
	up to isomorphism.
	
	For any \(z\in E\), \(\mathcal F_z \simeq \mathcal O_{\widetilde X,z}\) is locally free of rank one.
	Since \(\widetilde X\) is noetherian, \(\mathcal F\) is finitely
	presented, and its locally-free-of-rank-one locus is open
	\cite[Tag~0GWM]{StacksProject}. Let \(Z\subset\widetilde X\) be its
	closed complement. Then \(Z\cap E=\varnothing\).
	
	 The composite
	\(
	Z\longrightarrow\widetilde X\xrightarrow{\rho}\Spec(B)
	\)
	is proper, so \(\rho(Z)\) is closed
	\cite[Tag~01W6]{StacksProject}. If \(Z\) were nonempty, the nonempty
	closed subset \(\rho(Z)\) of the local scheme \(\Spec(B)\) would
	contain its closed point. Its inverse image would then meet the
	closed fibre \(E\), contradicting \(Z\cap E=\varnothing\).
	Therefore \(Z=\varnothing\), and \(\mathcal F\) is invertible.
	Essential surjectivity and full faithfulness in Grothendieck existence
	therefore give
	\[
	\Pic(\widetilde X)
	\xrightarrow{\sim}
	\varprojlim_n\Pic(E^{(n)})
	\xrightarrow{\sim}
	\Pic(E)
	\xrightarrow{\sim}
	\mathbf Z^2.
	\]

	Put \(\widetilde X^\circ:=\widetilde X\setminus E\).  By
	\cite[Tag~02OS]{StacksProject} the morphism \(\rho\) restricts to an
	isomorphism
	\(\widetilde X^\circ\simeq\Spec(B)\setminus\{\mathfrak m_B\}\); as noted
	at the beginning of this argument, these schemes are regular.  The
	scheme \(\widetilde X\) is integral \cite[Tag~02ND]{StacksProject},
	noetherian and regular, and so is its open subscheme
	\(\widetilde X^\circ\); their local rings are unique factorization
	domains \cite[Tag~0AG0]{StacksProject}, whence
	\(\Pic\xrightarrow{\ \sim\ }\Cl\) for both
	\cite[Tag~0BE9]{StacksProject}.

	We use the presentation of \(\Cl\) by Weil divisors modulo principal
	divisors \cite[Tags~0BE2, 0BE3 and~0BE4]{StacksProject}.  Taking closures
	identifies the prime divisors of \(\widetilde X^\circ\) with the prime
	divisors of \(\widetilde X\) meeting \(\widetilde X^\circ\), so
	restriction
	\(\operatorname{Div}(\widetilde X)\to\operatorname{Div}(\widetilde X^\circ)\)
	is surjective with kernel free on the prime divisors contained in \(E\),
	namely \(E_1\) and \(E_2\).  As \(\widetilde X^\circ\) is dense in
	\(\widetilde X\) and orders of vanishing are computed at generic points,
	this restriction carries \(\operatorname{div}(f)\) to
	\(\operatorname{div}(f)\) for every \(f\) in the common function field.
	Hence
	\(\mathbf ZE_1\oplus\mathbf ZE_2\to
	\Pic(\widetilde X)\to\Pic(\widetilde X^\circ)\to0\)
	is exact, the first map sending \(E_j\) to
	\(\mathcal O_{\widetilde X}(E_j)\).
	
	The identification \(\Pic(\widetilde X)\simeq\mathbf Z^2\) obtained above
	is \(\mathcal M\mapsto\bigl(\deg_{E_1}(\mathcal M|_{E_1}),
	\deg_{E_2}(\mathcal M|_{E_2})\bigr)\), so by the definition of the
	intersection numbers it sends \(\mathcal O_{\widetilde X}(E_j)\) to
	\((E_j\cdot E_1,E_j\cdot E_2)\).  The first map of the exact sequence
	above is therefore given by the matrix
	\(M:=\begin{psmallmatrix}
		E_1\cdot E_1&E_2\cdot E_1\\E_1\cdot E_2&E_2\cdot E_2
	\end{psmallmatrix}
	=\begin{psmallmatrix}-2&1\\1&-2\end{psmallmatrix}\),
	whose \(j\)-th column is the image of \(E_j\).
	
	Write \(\epsilon_1,\epsilon_2\) for the images in
	\(\operatorname{coker}(M)\) of the standard basis of \(\mathbf Z^2\).
	The two columns of \(M\) impose the relations
	\(-2\epsilon_1+\epsilon_2=0\) and \(\epsilon_1-2\epsilon_2=0\); the first
	gives \(\epsilon_2=2\epsilon_1\), and substituting it into the second
	gives \(3\epsilon_1=0\).  Thus \(\operatorname{coker}(M)\) is cyclic of
	order dividing three, generated by \(\epsilon_1\).  It is exactly of
	order three, since the surjection \(\mathbf Z^2\to\mathbf Z/3\),
	\((a,b)\mapsto a+2b\), kills both columns of \(M\).  Consequently
	\(\Pic(\widetilde X^\circ)\simeq\operatorname{coker}(M)\simeq\mathbf Z/3\).
	
	Finally, no prime divisor of \(\Spec(B)\) is contained in
	\(\{\mathfrak m_B\}\), so the same restriction argument, applied to
	\(\Spec(B)\setminus\{\mathfrak m_B\}\subset\Spec(B)\), gives
	\(\Cl(B)\xrightarrow{\ \sim\ }\Cl(\widetilde X^\circ)
	=\Pic(\widetilde X^\circ)\simeq\mathbf Z/3\).  Since \([P]\) has exact
	order three, it is a generator.

	We have $B[1/3]=B[1/\pi]$.  Restriction of Weil divisors to
	$D(\pi)$ is surjective, and its kernel is generated by the height-one
	primes containing $\pi$.  The only such prime is $(\pi)$, whose class is
	zero because it is the principal divisor of $\pi$.  The presentation of
	the divisor class group by Weil divisors and principal divisors therefore
	gives
	\(
	\Cl(B)\xrightarrow{\ \sim\ }\Cl(B[1/3])
	\)
	\cite[Tags~0BE2, 0BE3 and~0BE4]{StacksProject}.  The ring $B[1/3]$ is regular, so
	\(
	\Pic(T)\xrightarrow{\ \sim\ }\Cl(B[1/3])
	\)
	\cite[Tags~0AG0 and~0BE9]{StacksProject}.  Hence $\Pic(T)\simeq\mathbf Z/3$ and
	$\mathcal L=P|_T$ is a generator.  Since $\operatorname{div}(y)=P+Q$, we have
	$Q|_T\simeq(P|_T)^{-1}=\mathcal L^{-1}$.
\end{proof}

\section{Construction of the generic-kernel class}

If $R$ is noetherian, write $G(R)$ for the $K$-theory of finitely generated
$R$-modules; if $R$ is regular, $K(R)\to G(R)$ is an equivalence.  
A finite ring map \(f\colon R\to S\) gives a pushforward
\(
f_*\colon G(S)\longrightarrow G(R)
\)
by restriction of scalars, while a flat map gives a pullback by
base change.

\begin{lemma}\label{lem:divisor-y}
	The element $y$ is prime in $A$, and there is an isomorphism of
	$V$-algebras
	\[
	A/yA\xrightarrow{\ \sim\ }O,\qquad x\longmapsto t.
	\]
\end{lemma}

\begin{proof}
	Since \(t=-1-2\zeta\) and \(2\in V^\times\), one has
	\(O=V[\zeta]=V[t]\); moreover \(t^2=-3\).  
	
	Since \(t=\zeta\pi\) lies in the maximal ideal of the complete local
	ring \(O\), substitution \(x\mapsto t\) defines a continuous
	\(V\)-algebra map \(V\llbracket x\rrbracket\to O\).  As \(t^2=-3\),
	it induces a surjection
	\(
	V\llbracket x\rrbracket/(x^2+3)\to O
	\),
	the surjectivity following from \(O=V[t]\).
	By Weierstrass division the source is free over \(V\) with basis
	\(1,x\), while \(1,t\) is a \(V\)-basis of \(O\).  The map is therefore
	an isomorphism.

	Hence \(A/yA\simeq O\) is a domain, so \(yA\) is
	prime.
\end{proof}

Since $O$ is commutative local, determinant identifies $K_1(O)$ with
$O^\times$ \cite[Lem.~III.1.4]{Weibel13KBook}, and the unit $\zeta$ has
order three.  Exactness of the coefficient sequence for $O$ therefore
supplies a class $\beta_0\in K_2(O;\mathbf Z/3)$ with
$\partial_O(\beta_0)=[\zeta]$.  Put
\begin{equation}\label{eq:beta-definition}
	\beta:=\frac12\bigl(\beta_0-\sigma^*\beta_0\bigr)\in K_2(O;\mathbf Z/3).
\end{equation}
Because $\sigma(\zeta)=\zeta^2$ and $\tfrac12=2$ in $\mathbf Z/3$, one has
$\partial_O(\beta)=2\bigl([\zeta]-2[\zeta]\bigr)=-2[\zeta]=[\zeta]$ and
$\sigma^*\beta=-\beta$.

Let $i\colon\Spec(A/yA)\to\Spec(A)$ be the closed immersion and
$i_*\colon G(A/yA;\mathbf Z/3)\to G(A;\mathbf Z/3)$ the induced pushforward.
Using \cref{lem:divisor-y} to read $\beta$ as a class on $\Spec(A/yA)$, set
\begin{equation}\label{eq:a-definition}
	a:=i_*(\beta)\in G_2(A;\mathbf Z/3)=K_2(A;\mathbf Z/3).
\end{equation}

\begin{lemma}\label{lem:split-fibre}
	Write \(K_O:=\Frac(O)\).  The ring \(R=B[1/3]\) is flat over \(A\), and
	\(
	R/yR\xrightarrow{\ \sim\ }K_O\times K_O ,
	\)
	the two factors being $R/PR$ and $R/QR$.  Under \cref{lem:divisor-y} and
	the identifications $B/P\simeq O\simeq B/Q$ induced by $O\subset B$, the
	map $A/yA\to R/yR$ is $\lambda\mapsto\bigl(\lambda,\sigma(\lambda)\bigr)$.
\end{lemma}
\begin{proof}
	By \cref{prop:geometry} the ring $B$ is finite free over $A$, hence flat,
	and $R$ is a localization of $B$; so $R$ is flat over $A$.  Setting
	$y=0$ in \eqref{eq:B-definition} gives
	$B/yB\simeq O\llbracket x\rrbracket/\bigl((x-t)(x+t)\bigr)$.  Since
	$t^2=-3$, the element $2t$ is a unit after inverting $3$, so $(x-t)$ and
	$(x+t)$ are comaximal in $(B/yB)[1/3]$ and the Chinese remainder theorem
	gives
	\[
	R/yR\simeq
	\bigl(O\llbracket x\rrbracket/(x-t)\bigr)[1/3]\times
	\bigl(O\llbracket x\rrbracket/(x+t)\bigr)[1/3]
	\simeq K_O\times K_O ,
	\]
	the two projections sending $x$ to $t$ and to $-t$.  As $u=x-t$ and
	$v=x+t$, the first factor is $R/PR$ and the second is $R/QR$.  Finally
	$A/yA$ is generated over $V$ by the class of $x$, which corresponds to
	$t\in O$ under \cref{lem:divisor-y}, and $\sigma(t)=-t$; so the map
	$A/yA\to R/yR$ is $\lambda\mapsto(\lambda,\sigma\lambda)$.
\end{proof}

\begin{proposition}\label{prop:class}
	The class $a$ is nonzero, its image in $K_2(F;\mathbf Z/3)$ is zero, and
	$\partial_A(a)=0$.
\end{proposition}
\begin{proof}
	Since $y$ is a nonzerodivisor, Quillen's localization sequence
	\[
	G(A/yA;\mathbf Z/3)\xrightarrow{\ i_*\ }
	G(A;\mathbf Z/3)\longrightarrow
	G(A[1/y];\mathbf Z/3)
	\]
	is exact \cite[\S7]{Quillen73}, so $a$ dies in $G_2(A[1/y];\mathbf Z/3)$.
	As $F$ is a localization of $A[1/y]$, and as $A$ and $F$ are regular so
	that the resolution theorem identifies the $G$-groups occurring here with
	their $K$-groups \cite[Thm.~V.3.3]{Weibel13KBook}, the image of $a$ in
	$K_2(F;\mathbf Z/3)$ is zero.
	
	We prove that $a\ne0$.  Let $g\colon A\to R$ be as in
	\cref{lem:split-fibre}.  For a finite $A/yA$-module $M$ one has
	$M\otimes_AR=M\otimes_{A/yA}(R/yR)$, and both functors are exact because
	$R$ is flat over $A$; with \cref{lem:split-fibre} this gives
	\[
	g^*i_*=i_{P*}\varphi_P^*+i_{Q*}\varphi_Q^* ,
	\]
	where \(i_P,i_Q\) are the closed immersions of the two points of \(T\)
	cut out by \(P\) and \(Q\),	while
	\(\varphi_P\colon O\hookrightarrow K_O\) is the standard inclusion and
	\(\varphi_Q:=\varphi_P\circ\sigma\); these are the two components of
	\(A/yA\simeq O\to K_O\times K_O\).

	Writing $\beta|_T$ and
	$\beta|_{K_O}$ for the images of $\beta$ along $O\subset B\to R$ and
	$O\subset K_O$, and using $\sigma^*\beta=-\beta$, we obtain
	\(
	g^*(a)=i_{P*}\bigl(\beta|_{K_O}\bigr)-i_{Q*}\bigl(\beta|_{K_O}\bigr).
	\)
	Both identifications $B/P\simeq O\simeq B/Q$ restrict to the identity on
	$O$, so $i_P^*(\beta|_T)=i_Q^*(\beta|_T)=\beta|_{K_O}$, and the
	projection formula \cite[Thm.~V.3.12]{Weibel13KBook} gives
	$g^*(a)=\beta|_T\cdot\bigl([R/PR]-[R/QR]\bigr)$.

	Put \(\mathcal L:=P|_T\).  The exact sequences
	\(0\to\mathcal L\to R\to R/PR\to0\) and
	\(0\to Q|_T\to R\to R/QR\to0\) give
	\(
	[R/PR]-[R/QR]=[Q|_T]-[\mathcal L]
	\)
	in
	\(
	K_0(T).
	\)
	For a connected one-dimensional regular noetherian ring, rank and
	determinant induce an isomorphism
	\(K_0(T)\simeq\mathbf Z\oplus\Pic(T)\)
	\cite[Cor.~II.2.6.3]{Weibel13KBook}, under which this difference
	corresponds to
	\((0,Q|_T\otimes\mathcal L^{-1})\).
	By \cref{prop:geometry},
	\(Q|_T\simeq\mathcal L^{-1}\) and
	\(\mathcal L^{\otimes3}\simeq\mathcal O_T\), so
	\(Q|_T\otimes\mathcal L^{-1}\simeq
	\mathcal L^{-2}\simeq\mathcal L\)
	and therefore
	\(
	g^*(a)=\beta|_T\bigl([\mathcal L]-1\bigr).
	\)
	
	We check the hypotheses of \cref{lem:bott-picard} for $T$.  By
	\cref{prop:geometry} it is integral, regular, noetherian and
	one-dimensional, and the class of $\mathcal L$ generates $\Pic(T)\simeq\mathbf Z/3$, hence is nonzero in $\Pic(T)/3$. The ring
	\(\Gamma(T,\mathcal O_T)\) contains \(\zeta\), and \(3\) is invertible
	on \(T\). If \(K_O=\Frac(O)\), then \(R\) is a
	regular noetherian $K_O$-algebra; the field $K_O$ is perfect because it
	has characteristic zero, so $R$ is geometrically regular over $K_O$
	\cite[Tags~05DU and~038V]{StacksProject}, and Popescu's theorem writes
	$R$ as a filtered colimit of smooth $K_O$-algebras
	\cite[Tag~07GC]{StacksProject}.  Finally
	$\partial_T(\beta|_T)=[\zeta]$ by naturality of the coefficient boundary.
	\cref{lem:bott-picard} therefore gives $g^*(a)\ne0$, whence $a\ne0$.

	It remains to compute the Bockstein.  Exactness of
	\eqref{eq:bockstein-sequence} shows that $\partial_A(a)\in K_1(A)[3]$,
	and determinant identifies \(K_1(A)\) with \(A^\times\)
	\cite[Lem.~III.1.4]{Weibel13KBook}.  Now
	\(
	B\otimes_A F=F\otimes_V O=F[\zeta]
	\)
	is a domain, being a localization of \(B\), and has \(F\)-dimension two
	because \(B\) is finite free of rank two over \(A\)
	(\cref{prop:geometry}).  A finite-dimensional domain over a field is a
	field, so \(F[\zeta]\) is quadratic over \(F\) and \(\zeta\notin F\).
	Also \(\zeta^2\notin F\), since otherwise
	\(\zeta=(\zeta^2)^2\) would belong to \(F\).  Thus \(F\), and hence
	\(A\), contains no nontrivial cube root of unity.  Therefore \(K_1(A)[3]=K_1(F)[3]=0\).
	The coefficient sequences identify
	\(K_2(A;\mathbf Z/3)\simeq K_2(A)/3\) and
	\(K_2(F;\mathbf Z/3)\simeq K_2(F)/3\); in particular,
	\(\partial_A(a)=0\).

\end{proof}

\begin{proof}[Proof of \cref{thm:main}]
The first paragraph of the proof of \cref{prop:geometry} shows that
\(A\) is a complete two-dimensional ramified regular local domain.
The required class is the class \(a\) constructed in
\cref{prop:class}.
\end{proof}

\section{Integral divisibility and possible odd-prime variants}

The counterexample is represented by an integral class that becomes
divisible by three over the fraction field.

\begin{corollary}\label{cor:saturation}
There is $c\in K_2(A)\setminus 3K_2(A)$ lifting $a$ and an element $d\in K_2(F)$ such that
$c|_F=3d$, but $d$ is not in the image of $K_2(A)$.  
\end{corollary}

\begin{proof}
Since $\partial_A(a)=0$, exactness of \eqref{eq:bockstein-sequence}
gives an integral lift $c$. Since \(a\ne0\), necessarily \(c\notin3K_2(A)\). Since $a|_F=0$, the same sequence over $F$
gives $c|_F=3d$ for some $d\in K_2(F)$.  Van der Kallen proved that
$K_2(A)\to K_2(F)$ is injective
\cite[main theorem]{VanDerKallen76KTwoSurface}.  If $d=e|_F$ for some
$e\in K_2(A)$, then $c-3e$ would vanish over $F$ and hence vanish in
$K_2(A)$; taking the image in $K_2(A;\mathbf Z/3)$ would give $a=0$, a contradiction.
\end{proof}

We record a proposed odd-prime extension in degree \(p-1\);
we do not prove it here.

\begin{remark}
\label{rem:odd-prime}
The same method is expected to yield the following variants.

Let $p$ be odd, and $V=W(\overline{\mathbf F}_p)$. Let
$\chi\colon\Gal(V[\zeta_p]/V)\to\mathbf F_p^\times$ be the cyclotomic character, and put
$s=(p-1)/2$.  Choose a uniformizer $t\in V[\zeta_p]^{\ker(\chi^s)}$ such that
$\sigma(t)=\chi^s(\sigma)t$ for any $\sigma\in \Gal(V[\zeta_p]/V)$, and set $d=t^2\in V$. Then
\(
 A_p:=V\llbracket x,y\rrbracket/(x^2-d-y^p)
\)
is a complete two-dimensional ramified regular local domain.  The expected conclusion is that there exists
\[
 0\ne a_p\in\ker\bigl(K_{p-1}(A_p;\mathbf Z/p)
             \longrightarrow K_{p-1}(\Frac(A_p);\mathbf Z/p)\bigr).
\]

\end{remark}

\begin{remark}\label{rem:elementary}

	For $p=3$, elementary computations using symbols could be used to bypass \cref{sec:motivic-detection}.\footnote{Put
		\(g:=-(1+x)/2\) and \(s:=\{g,y\}\in K_2(A[1/y])\).
		Since \(g\bmod y=\zeta\), localization lifts \(3s\) to \(c\in K_2(A)\).
		If \(c=3e\), then \((s-e)|_L\in K_2(L)[3]\), where \(L=\Frac(B)\),
		equals \(\{\zeta,h\}\) for some \(h\in L^\times\) by
		\cite[Thm.~III.6.8]{Weibel13KBook}.
		As \(e|_L\) is unramified on \(T\), tame symbols give
		\(\operatorname{div}_T(h)\equiv P-Q\pmod3\), contradicting
		\cref{prop:geometry}. Thus \(c\bmod3\) is a nonzero generic-kernel class.}

	We have kept the
	motivic argument: it exhibits the class as a weight-two phenomenon, and it
	is the one that should generalize to the higher-degree variants in \(K_{p-1}\)
	described in \cref{rem:odd-prime} for odd $p$.

\end{remark}

\subsection*{Acknowledgments}

The author is grateful to Frédéric Déglise for introducing him to the Gersten conjecture
in January 2018 and for many subsequent discussions around this problem.

\medskip
\noindent\textbf{Use of AI-assisted tools}.
Large language models were used as interactive tools for expanding proof
sketches, checking consistency, improving exposition, and verifying
citations to the literature.  All resulting
suggestions were checked by the author, who assumes full responsibility
for the mathematical statements, proofs, and references.

\bibliographystyle{amsalpha}
\bibliography{FELD_finite_coeff_Gersten_injectivity_fails_in_ramified_mixed_char}

\end{document}